\documentclass[final,1p,times]{elsarticle}
\usepackage{amsfonts, amstext, amsmath, amssymb, mathrsfs, amsthm}

\newcommand{\R}{\mathbb{R}}

\theoremstyle{plain}
\newtheorem{theorem}{Theorem}

\newtheorem{proposition}[theorem]{Proposition}

\theoremstyle{definition}
\newtheorem{assumption}[theorem]{Assumption}

\begin{document}

\begin{frontmatter}

\title{Bi-stochastic kernels via asymmetric affinity
  functions\tnoteref{t1}}
\tnotetext[t1]{{\it Applied and Computational Harmonic
    Analysis}, volume 35, number 1, pages 177-180, 2013. arXiv:1209.0237.}

\author{Ronald R. Coifman}
\ead{coifman@math.yale.edu}

\author{Matthew J. Hirn\corref{cormatt}}
\ead{matthew.hirn@yale.edu}
\ead[url]{www.math.yale.edu/$\sim$mh644}

\cortext[cormatt]{Corresponding author}

\address{
Yale University \\
Department of Mathematics \\
P.O. Box 208283 \\
New Haven, Connecticut 06520-8283 \\
USA
}

\begin{abstract}
In this short letter we present the construction of a bi-stochastic
kernel $p$ for an arbitrary data set $X$ that is derived from an asymmetric affinity
function $\alpha$. The affinity function $\alpha$ measures the
similarity between points in $X$ and some reference set $Y$. Unlike
other methods that construct bi-stochastic kernels via some
convergent iteration process or through solving an optimization
problem, the construction presented here is quite simple. Furthermore,
it can be viewed through the lens of out of sample extensions, making
it useful for massive data sets.
\end{abstract}

\begin{keyword}
bi-stochastic kernel; Nystr\"{o}m extension
\end{keyword}

\end{frontmatter}

\section{Introduction}

Given a positive, symmetric kernel (matrix) $k$, the question of
how to construct a bi-stochastic kernel derived from $k$ has been of
interest in certain applications such as data clustering. Various
algorithms for this task exist. One of the best known is the
Sinkhorn-Knopp algorithm \cite{sinkhorn:doublyStochastic1967}, in
which one alternately normalizes the rows and columns of $k$ to sum to
one. A symmetrization of this algorithm is given in
\cite{zass:probCluster2005}, and is subsequently used to cluster
data. In both cases, an infinite number of iterations are needed for
the process to converge to a bi-stochastic matrix. In another
application of data clustering, the authors in
\cite{wang:clusterBiStochastic2011} solve a quadratic programming
problem to obtain what they call the Bregmanian bi-stochastication of
$k$. Common to these algorithms and others is the complexity in
solving for (or approximating) the bi-stochastic matrix.

Also related to the goal of organizing data, over the last decade we
have seen the development of a class of research that utilizes
nonlinear mappings into low dimensional spaces in order to organize
potentially high dimensional data. Examples include locally linear
embedding (LLE) \cite{roweis:lle2000}, ISOMAP
\cite{tenenbaum:isomap2000}, Hessian LLE \cite{donoho:hessianlle2003},
Laplacian eigenmaps \cite{belkin:laplacianEigen2003}, and diffusion
maps \cite{coifman:diffusionMaps2006}. In many applications, these
data sets are not only high dimensional, but massive. Thus there has
been the need to develop complementary methods by which these
nonlinear mappings can be computed efficiently. The Nystr\"{o}m
extension is one early such example; in \cite{bengio:outOfSample2004}
several out of sample extensions are given for various nonlinear
mappings, while \cite{coifman:geometricHarmonics2006} utilizes
geometric harmonics to extend empirical functions.

In this letter we present an extremely simple bi-stochastic kernel
construction that can also be implemented to handle massive data
sets. Let $X$ be the data set. The entire construction is derived not
from a kernel on $X$, but rather an asymmetric affinity function
$\alpha: X \times Y \rightarrow \R$ between the given
data and some reference set $Y$. The key to realizing the
bi-stochastic nature of the derived kernel is to apply the correct
weighted measure on $X$. The eigenfunctions (or eigenvectors) of
this bi-stochastic kernel on $X$ are also easily computable via a
Nystr\"{o}m type extension of the eigenvectors of a related kernel on
$Y$. Since the reference set can usually be taken to much smaller than
the original data set, these eigenvectors are simple to compute. 

\section{A simple bi-stochastic kernel construction}

We take our data set to be a measure space $(X,
\mu)$, in which $\mu$ represents the distribution of the points. We
also assume that we are given, or able to compute, a finite reference
set $Y \triangleq \{y_1, \ldots, y_n\}$. Note that one can take $X$ to
be discrete or finite as well; in particular, one special case is
when $X = Y$ and $\mu = \frac{1}{n}\sum_{i=1}^n \delta_{y_i}$. 

\subsection{Affinity functions and densities}

Let $\alpha: X \times Y \rightarrow \R$ be a positive affinity
function that measures the similarity between the data set $X$ and the
reference set $Y$. Larger values of $\alpha (x,y_i)$ indicate that the
two data points are very similar, while those values closer to zero
imply that $x$ and $y_i$ are quite different. The
  function $\alpha$ serves as a generalization of the traditional
  kernel function $k: X \times X \rightarrow \R$, in which $k(x,x')$
  measures the similarity between two points $x,x' \in X$ (just as
  with $\alpha$, the larger $k(x,x')$, the more similar the two
  points). Kernel functions have been successfully used in applied mathematics and
  machine learning for various data driven tasks. Certain kernel
  functions can be viewed as an inner product $k(x,x') = \Phi(x) \cdot
  \Phi(x')$, after the data set $X$ has been mapped nonlinearly into a
  higher dimensional space via $\Phi$. The idea is that if the kernel
  $k$ is constructed carefully, than the mapping $\Phi$ will arrange
  the data set $X$ so that certain tasks (such as data clustering) can
  be done more easily (say using linear methods). We show that the
  more general function $\alpha$ can be used similarly, with the added
benefits of deriving a bi-stochastic kernel that is amenable to an out of
sample Nystr\"{o}m type extension.

We derive two densities from the affinity function $\alpha$, which we
shall then use to normalize it. The first of these is the density
$\Omega: X \rightarrow \R$ on the data set $X$; we take it as
\begin{equation*}
\Omega(x) \triangleq \sum_{i=1}^n \alpha(x,y_i), \quad \text{for all }
x \in X.
\end{equation*}
We also have a density $\omega: Y \rightarrow \R$ on the reference
set, which we define as:
\begin{equation*}
\omega(y_i) \triangleq \left( \int\limits_X \alpha(x, y_i) \, \Omega(x) \,
  d\mu(x) \right)^{\frac{1}{2}}, \quad \text{for all } y_i \in Y.
\end{equation*}

\begin{assumption} \label{assumptions on alpha}
We make the following simple assumptions concerning $\alpha$:
\begin{enumerate}
\item
For each $y_i \in Y$, the function $\alpha(\cdot,y_i): X \rightarrow
\R$ is square integrable, i.e.,
\begin{equation} \label{eqn: alpha assumption}
\alpha(\cdot,y_i) \in L^2(X,\mu).
\end{equation}
\item
The densities $\Omega$ and $\omega$ are finite and strictly positive:
\begin{align} 
0 < \Omega(x) < \infty, &\quad \text{for all } x \in X, \label{eqn: omega x
  assumption} \\
0 < \omega(y_i) < \infty, &\quad \text{for all } y_i \in Y. \label{eqn:
  omega y assumption}
\end{align}
\end{enumerate}
\end{assumption}

The $L^2$ integrability condition is necessary to make sure that
functions and operators related to $X$ make sense. The upper bounds on
the densities simply put a finite limit on how close
any point $x$ or $y_i$ is to either $Y$ or $X$,
respectively. Meanwhile, the lower bounds state that each point in $X$
has some relation to the reference set $Y$, and likewise that each
reference point $y_i$ has some similarity to at least part of
$X$.

Using $\alpha$ and the two densities $\Omega$ and $\omega$, we define
a normalized affinity $\beta: X \times Y \rightarrow \R$ as
\begin{equation*}
\beta(x,y_i) \triangleq \frac{\alpha(x,y_i)}{\Omega(x) \, \omega(y_i)}, \quad
\text{for all } (x,y_i) \in X \times Y.
\end{equation*}

From this point forward we will use the weighted measure $\Omega^2\mu$
on $X$. This measure is the ``correct'' measure in the sense that it
is the measure for which we can define a bi-stochastic kernel
on $X$ in a natural, simple way. Using Assumption \ref{assumptions on
  alpha}, one can easily show that for each $y_i \in Y$, the function
$\beta(\cdot,y_i): X \rightarrow \R$ is well behaved under this
measure:
\begin{equation*}
\beta(\cdot, y_i) \in L^2(X, \Omega^2 \mu), \quad \text{for all } y_i \in Y.
\end{equation*}

We note that affinity functions similar to $\beta$ were first considered in
\cite{kushnir:anisotropicDiffEarth2011} in the context of out of
sample extensions for independent components analysis (ICA). It has
also been utilized in \cite{haddad:filteringReferenceSet2011} in the
context of filtering. The connection to bi-stochastic kernels, though,
has until now gone unnoticed. 

\subsection{The bi-stochastic kernel}

To construct the bi-stochastic kernel on $X$ we utilize the $\beta$
affinity function. Let $p: X \times X \rightarrow \R$ denote the
kernel, and define it as:
\begin{align*}
p(x,x') &\triangleq \langle \beta(x, \cdot), \beta(x', \cdot)
\rangle_{\R^n} \\
&= \sum_{i=1}^n \beta(x,y_i) \, \beta(x',y_i), \quad \text{for all } (x,x') \in X \times X.
\end{align*}
The following proposition summarizes the main properties of $p$.
\begin{proposition} \label{thm: p is bi-stochastic}
If Assumption \ref{assumptions on alpha} holds, then:
\begin{enumerate}
\item
The kernel $p$ is square integrable under the weighted measure
$\Omega^2\mu$, i.e.,
\begin{equation*}
p \in L^2(X \times X, \Omega^2 \mu \otimes \Omega^2 \mu).
\end{equation*}

\item
The kernel $p$ is bi-stochastic under the weighted measure
$\Omega^2\mu$, i.e.,
\begin{equation*}
\int\limits_X p(x,x') \, \Omega(x')^2 \mu(x') = \int\limits_X p(x',x)
\, \Omega(x')^2 \mu(x') = 1, \quad \text{for all } x \in X.
\end{equation*}
\end{enumerate}
\end{proposition}

\begin{proof}
We begin with the $L^2$ integrability condition. Using the definition of
$p$, expanding its $L^2$ norm, and applying H\"{o}lder's theorem gives:
\begin{align*}
\|p\|_{L^2(X \times X, \Omega^2\mu \otimes
  \Omega^2\mu)}^2 &= \int\limits_X \int\limits_X \sum_{i,j=1}^n
\beta(x,y_i) \, \beta(x',y_i) \, \beta(x,y_j) \, \beta(x',y_j) \,
\Omega(x')^2 \, \Omega(x)^2 \, d\mu(x) d\mu(x') \\
&\leq n^2 \max_{y_i \in Y} \| \, \beta(\cdot,y_i)\|_{L^2(X,\Omega^2\mu)}^4 \\
&< \infty.
\end{align*}

Now we show that $p$ is bi-stochastic:
\begin{align*}
\int\limits_X p(x,x') \, \Omega(x')^2 \, d\mu(x') &= \int\limits_X
\sum_{i=1}^n \frac{\alpha(x,y_i) \, \alpha(x',y_i)}{\Omega(x) \,
  \Omega(x') \, \omega(y_i)^2} \, \Omega(x')^2 \, d\rho(y) d\mu(x') \\
&= \sum_{i=1}^n \frac{\alpha(x,y_i)}{\Omega(x) \, \omega(y_i)^2}
\int\limits_X \alpha(x',y_i) \, \Omega(x') \, d\mu(x') \\
&= \sum_{i=1}^n \frac{\alpha(x,y_i)}{\Omega(x)} \\
&= 1.
\end{align*}
Since $p$ is clearly symmetric, this completes the proof.
\end{proof}

Since $p \in L^2(X \times X, \Omega^2\mu \otimes \Omega^2\mu)$, one
can define the integral operator $P: L^2(X,\Omega^2\mu) \rightarrow
L^2(X,\Omega^2\mu)$ as:
\begin{equation*}
(Pf)(x) \triangleq \int\limits_X p(x,x') f(x') \, \Omega(x')^2 \,
d\mu(x'), \quad \text{for all } f \in L^2(X,\Omega^2\mu).
\end{equation*}
Given the results of Proposition \ref{thm: p is bi-stochastic}, we see
that $P$ is a Hilbert-Schmidt, self-adjoint, diffusion operator. In
terms of data organization and clustering, it is usually the
eigenfunctions and eigenvalues of $P$ that are of interest; see, for
example, \cite{coifman:diffusionMaps2006}. The fact that $P$ is
bi-stochastic though, as opposed to merely row stochastic, could make it
particularly interesting for these types of applications in which
$I-P$ is used as an approximation of the Laplacian.

\subsection{A Nystr\"{o}m type extension}

The affinity $\beta$ can also be used to construct an $n \times n$
matrix $A$ as follows:
\begin{align*}
A[i,j] &\triangleq \langle \beta(\cdot,y_i), \beta(\cdot,y_j)
\rangle_{L^2(X,\Omega^2\mu)} \\
&= \int\limits_X \beta(x,y_i) \, \beta(x,y_j) \, \Omega(x)^2 \,
d\mu(x), \quad \text{for all } i,j = 1, \ldots, n.
\end{align*}
The matrix $A$ is useful for computing the eigenfunctions and
eigenvalues of $P$. The following proposition is simply an
interpretation of the singular value decomposition (SVD) in this context.
\begin{proposition}
If Assumption \ref{assumptions on alpha} holds, then:
\begin{enumerate}
\item
Let $\lambda \in \R \setminus \{0\}$. Then $\lambda$ is an eigenvalue
of $P$ if and only if it is an eigenvalue of $A$. 
\item
Let $\lambda \in \R \setminus \{0\}$. If $\psi \in
L^2(X,\Omega^2\mu)$ is an eigenfunction of $P$ with eigenvalue
$\lambda$ and $v \in \R^n$ is the corresponding eigenvector of $A$, then:
\begin{align} 
\psi(x) &= \frac{1}{\sqrt{\lambda}} \sum_{i=1}^n \beta(x,y_i) \,
v[i], \nonumber \\
v[i] &= \frac{1}{\sqrt{\lambda}} \int\limits_X \beta(x,y_i) \, \psi(x) \,
\Omega(x)^2 \, d\mu(x). \nonumber
\end{align}
\end{enumerate}
\end{proposition}

\section{Acknowledgements}

This research was supported by Air Force Office of Scientific Research
STTR FA9550-10-C-0134 and by Army Research Office MURI
W911NF-09-1-0383. We would also like to thank the anonymous reviewer
for his or her helpful comments and suggestions.

\bibliographystyle{model1-num-names}
\bibliography{DiffusionGeometry}

\end{document}